\documentclass[11pt,a4paper,reqno]{amsart}
\usepackage[utf8]{inputenc}
\usepackage{amsmath}
\usepackage{amsthm}
\usepackage{mathrsfs}
\usepackage{amssymb}
\usepackage{enumitem}
\usepackage{cite}
\usepackage{mathtools}
\usepackage{subfiles}
\usepackage[usenames,dvipsnames]{xcolor}
\usepackage{comment}
\usepackage{pdfpages}
\usepackage{hyperref}
\hypersetup{
    colorlinks=true,
    hidelinks,
    linkcolor=blue,
    citecolor=LimeGreen,
    pdftitle={Gelfand-Shilov Spaces},
    pdfborder={0 0 0},
    allbordercolors={1 1 1}
}

\numberwithin{equation}{section}
\theoremstyle{definition}
\newtheorem{thm}{Theorem}[section]
\newtheorem{lemma}[thm]{Lemma}
\newtheorem{cor}[thm]{Corollary}
\newtheorem{prop}[thm]{Proposition}

\theoremstyle{definition}
\newtheorem{definition}[thm]{Definition}
\newtheorem{remark}[thm]{Remark}
\DeclareMathOperator*{\indlim}{ind\,lim}
\title{Fourier characterizations and non-triviality of Gelfand-Shilov spaces, with applications to Toeplitz operators}
\author{Albin Petersson}
\address{Department of Mathematics,
Linn{\ae}us University, V{\"a}xj{\"o}, Sweden}
\email{albin.petersson@lnu.se}
\date{\today}
\begin{document}
\begin{abstract}
    We examine properties of Gelfand-Shilov spaces $S_s$, $S^\sigma$, $S_s^\sigma$,  $\Sigma_s$, $\Sigma^\sigma$ and $\Sigma_s^\sigma$. These are spaces of smooth functions where the functions or their Fourier transforms admit sub-exponential decay. It is determined that $\Sigma_s^\sigma$ is nontrivial if and only if $s+\sigma > 1$. We find growth estimates on functions and their Fourier transforms in the one-parameter spaces, and we obtain characterizations in terms of estimates of short-time Fourier transforms for these spaces and their duals. Additionally, we determine conditions on the symbols of Toeplitz operators under which the operators are continuous on one-parameter spaces.
\end{abstract}
\maketitle

\section{Introduction}\label{sec:intro}
The Gelfand-Shilov spaces were first introduced as a useful set of functions for the study of Cauchy problems in partial differential equations. These functions are convenient in this setting because of their smoothness, and because of the conditions of regularity imposed on them. For instance, some partial differential equations are ill-posed in the Schwartz space $\mathscr{S}$ or its dual $\mathscr{S}'$, the space of tempered distributions (see e.g. \cite[p.~160-163]{hormander} for notation), but are well-posed in suitable Gelfand-Shilov spaces. One such example is the Euler-Tricomi equation $D_t^2 f + t D_x^2 f = 0$, and another example may be found in \cite{lewy}. The fact that some partial differential equations are only well-posed in Gelfand-Shilov spaces exemplifies the need to determine properties of functions in those spaces. Gelfand-Shilov spaces can also be useful in the study of pseudo-differential operators \cite{cappiello}, which in turn have uses in, for instance, quantum theory \cite{quantum} and signal processing \cite{signal}.

The Gelfand-Shilov spaces $S_s^\sigma$, $S_s$ and $S^\sigma$ of Roumieu type (cf. \cite{eijndhoven,gel,chung}) and $\Sigma_s^\sigma$, $\Sigma_s$ and $\Sigma^\sigma$ of Beurling type (cf. \cite{pilip}) can be considered as refinements of the Schwartz space $\mathscr{S}$, where we impose analyticity-like smoothness conditions. The strength of these conditions depend on the parameters $s$ and $\sigma$. The smaller $s$ is the faster the functions must vanish at infinity, and smaller $\sigma$ impose stronger conditions on the growth of the derivatives (meaning the Fourier transform vanishes faster). In the one-parameter spaces, functions have sub-exponential decay and their Fourier transforms tend to zero faster than the reciprocal of any polynomial, or vice versa. In the two-parameter spaces, both the functions and their Fourier transforms have sub-exponential decay. If $s$ and $\sigma$ are sufficiently small, the only function found in $S_s^\sigma$ or $\Sigma_s^\sigma$ is $f(x) \equiv 0$, and the spaces are considered trivial. There are more general Gelfand-Shilov spaces, such as the $S_{M_p}^{N_p}$-spaces whose properties are explored in \cite{chung}, for instance. 

In this paper, we are mostly interested in discussing the properties of the one-parameter spaces $S_s$ and $S^\sigma$, their duals $(S_s)'$ and $(S^\sigma)'$, as well as the corresponding spaces $\Sigma_s$ and $\Sigma^\sigma$ and their duals. More specifically, we establish growth estimates on elements in these spaces and their Fourier transforms. Additionally, we find estimates involving the short-time Fourier transform which provides an alternate characterization of Gelfand-Shilov spaces. Such estimates exist for the two-parameter spaces (cf. \cite{grochenig}) and here we extend characterizations of this type to one-parameter spaces. We find that the short-time Fourier transform admits sub-exponential decay in one parameter, and tends to zero faster than reciprocals of polynomials in the other. Corresponding estimates are found for the duals of one-parameter spaces as well. We also examine Toeplitz operators on these one-parameter spaces, where the symbol $a(x,\xi)$ of the operator lies in different one-parameter spaces in each variable. We find conditions such that the Toeplitz operator is continuous on $S_s$, $S^\sigma$ and their respective duals.

We also determine when the two-parameter spaces are nontrivial. These results are well-known for $S_s^\sigma$-spaces, but for the $\Sigma_s^\sigma$-spaces, we find that the space is nontrivial if and only if $s+\sigma > 1$, as opposed to the condition $s+\sigma \geq 1$, $(s,\sigma)\neq (\frac12,\frac12)$ often cited in other works (cf. \cite{toft2}).  This result, which was initially suggested by Andreas Debrouwere, directly contradicts versions of this result in previous works.

The paper is structured as follows. In Section \ref{sec:prelim}, we introduce notations, definitions and preliminary propositions regarding Gelfand-Shilov spaces necessary to obtain results in subsequent sections. These preliminary results can either be found in \cite{chung,eijndhoven,gel} or are simple enough to be left as an exercise for the reader. In Section \ref{sec:nontriv}, we determine for which $s$ and $\sigma$ the space $\Sigma_s^\sigma$ is nontrivial. In Section \ref{sec:stft}, we obtain growth estimates for the short-time Fourier transform of functions in $S_s$, $S^\sigma$ and $\Sigma_s$, $\Sigma^\sigma$. In Section \ref{sec:dual}, we show how these results can be used to characterize the duals of these one-parameter spaces via the short-time Fourier transform as well. Lastly, in Section \ref{sec:top}, we find conditions on the symbol of Toeplitz operators so that the operator is continuous on one-parameter spaces and their duals.

\section{Preliminaries}\label{sec:prelim}
We begin by defining the spaces we will devote the most attention to in this paper. These are the so-called Gelfand-Shilov spaces. There is a clear and intuitive correspondence between the spaces of Roumieu type and those of Beurling type and the order in which the definitions are listed is meant to highlight this correspondence.
\begin{definition}
 Suppose $s,\sigma>0$. 
 \begin{enumerate}[label=(\roman*)]
     \item $S_s(\mathbb{R}^n)$ consists of all $f\in C^\infty(\mathbb{R}^n)$ for which there is an $h>0$ such that
     \begin{equation}\label{Sscond}
     \sup_{x\in \mathbb{R}^n} | x^\alpha D^\beta f(x) | \leq C_\beta h^{|\alpha|} \alpha!^s, \quad \forall \alpha,\beta\in\mathbb{N}^n,
     \end{equation}
      where $C_\beta$ is a constant depending only on $\beta$. 
      
      \item $\Sigma_s(\mathbb{R}^n)$ consists of all $f\in C^\infty(\mathbb{R}^n)$ such that \eqref{Sscond} holds for every $h>0$, where $C_\beta = C_{h,\beta}$ depends on $h$ and $\beta$.
      
     \item $S^\sigma(\mathbb{R}^n)$ consists of all $f\in C^\infty(\mathbb{R}^n)$ for which
     \begin{equation}\label{Ssigmacond}
     \sup_{x\in \mathbb{R}^n} | x^\alpha D^\beta f(x) | \leq C_\alpha h^{|\beta|} \beta!^\sigma, \quad \forall \alpha,\beta\in\mathbb{N}^n,
     \end{equation}
     holds for some $h>0$, where $C_\alpha$ is a constant depending only on $\alpha$.
     
     \item $\Sigma^\sigma(\mathbb{R}^n)$ consists of all $f\in C^\infty(\mathbb{R}^n)$ such that \eqref{Ssigmacond} holds for every $h>0$, where $C_\alpha = C_{h,\alpha}$ depends on $h$ and $\alpha$.
     
     \item $S_s^\sigma(\mathbb{R}^n)$ consists of all $f\in C^\infty(\mathbb{R}^n)$ for which there are constants $h>0$ and $C>0$ such that
     \begin{equation}\label{GScond}
     \sup_{x \in \mathbb{R}^n}|x^\alpha D^\beta f(x) | \leq C h^{|\alpha + \beta|} \alpha!^s \beta!^\sigma, \quad \forall \alpha,\beta\in\mathbb{N}^n.
     \end{equation}
     
     \item $\Sigma_s^\sigma(\mathbb{R}^n)$ consists of all $f\in C^\infty(\mathbb{R}^n)$ such that \eqref{GScond} holds for all $h>0$, where $C=C_h$ depends only on $h$.
 \end{enumerate}
\end{definition}
Trivially, we see that $S_s^\sigma \subseteq S_s\cap S^\sigma$, $\Sigma_s^\sigma \subseteq \Sigma_s \cap \Sigma^\sigma$, and $\Sigma_s^\sigma \subseteq S_s^\sigma$ for all such $s$ and $\sigma$. In fact, we have $S_s^\sigma = S_s\cap S^\sigma$ (cf. \cite{eijndhoven}) and $\Sigma_s^\sigma=\Sigma_s\cap\Sigma^\sigma$ (this is a well-known result that follows by analogous arguments, but an explicit proof can be found in \cite{albin} for instance).
\begin{definition}
 Let $\mathscr{F}$ denote the Fourier transform given by
\begin{equation*}
    (\mathscr{F}f)(\xi) = \hat{f}(\xi) = \frac{1}{(2\pi)^{n/2}} \int_{\mathbb{R}^n} f(x) e^{- i \langle x,\xi \rangle} \, dx,
\end{equation*}
and let $\mathscr{F}^{-1}$ be the corresponding inverse Fourier transform
\begin{equation*}
    (\mathscr{F}^{-1}f)(x) = \frac{1}{(2\pi)^{n/2}} \int_{\mathbb{R}^n} f(\xi) e^{i \langle x,\xi \rangle} \, d\xi.
\end{equation*}
If $f$ is a generalized function, then $\mathscr{F}$ denotes the adjoint operator of the Fourier transform defined above.
\end{definition}
Here we list some basic properties of Gelfand-Shilov spaces in the form of two propositions. The first proposition establishes sub-exponential decay of derivatives in Gelfand-Shilov spaces, and the second establishes how Fourier transforms work in Gelfand-Shilov spaces.  For both of the following two propositions, (a) can be found in \cite{eijndhoven,gel,chung}, and (b) follows by analogous arguments.
\begin{prop} \label{prop:Ssalt}
Suppose $s>0$ and $f\in C^\infty(\mathbb{R}^n)$. Then
\begin{enumerate}[label=(\alph*)]
    \item $f\in S_s(\mathbb{R}^n)$ if and only if there are constants $C_\beta,r>0$ such that
    \begin{equation}\label{eq:Ssalt}
    |D^\beta f(x)| \leq C_\beta e^{- r |x|^{1/s}}
    \end{equation}
    for all multi-indices $\beta$;
    \item $f\in\Sigma_s(\mathbb{R}^n)$ if and only if for every $r>0$
\begin{equation}\label{eq:Ssigalt}
    |D^\beta f(x)| \leq C_{r,\beta} e^{- r |x|^{1/s}}
\end{equation}
holds for all multi-indices $\beta$, where $C_{r,\beta}>0$ depends only on $r$ and $\beta$.

\end{enumerate}
\end{prop}
\begin{prop} \label{prop:FT}
Suppose $s,\sigma>0$. 
\begin{enumerate}[label=(\alph*)]
    \item If $s+\sigma\geq 1$, then $f\in S_s^\sigma(\mathbb{R}^n)$ if and only if $\hat{f}\in S_\sigma^s(\mathbb{R}^n)$. Moreover, $f\in S_s(\mathbb{R}^n)$ if and only if $\hat{f}\in S^s(\mathbb{R}^n)$.
    \item If $s+\sigma>1$, then $f\in\Sigma_s^\sigma(\mathbb{R}^n)$ if and only if $\hat{f}\in\Sigma_\sigma^s(\mathbb{R}^n)$. Moreover, $f\in \Sigma_s(\mathbb{R}^n)$ if and only if $\hat{f}\in \Sigma^s(\mathbb{R}^n)$.
\end{enumerate}
\end{prop}
We also include the following basic result.
\begin{prop}\label{prop:SsubSig}
If $s,\sigma >0$, $s<s_1$ and $\sigma<\sigma_1$, then $S_s^\sigma(\mathbb{R}^n) \subseteq \Sigma_{s_1}^{\sigma_1}(\mathbb{R}^n)$.
\end{prop}
We will now discuss the topology of Gelfand-Shilov spaces. To do this we need the following definition.
\begin{definition}\label{def:indproj}
Suppose $V_j$,  $j=0,1,2,\dots $, are Banach spaces, $$V= \bigcap_{j\geq 0} V_j$$ and $$W = \bigcup_{j\geq 0} V_j.$$
\begin{enumerate}[label=(\alph*)]
    \item Let $i_j:V\rightarrow V_j $ be inclusion maps. We say that the \emph{projective limit}  is the space $V$ with the smallest possible topology such that $i_j$ is continuous for all $j$. We write this as
    $$ V = \projlim_{j\geq 0} V_j. $$ 
    \item Suppose further that $V_j\hookrightarrow V_{j+1}$, meaning that $V_j$ is continuously embedded in $V_{j+1}$, and let $\tilde{i}_j:V_j \rightarrow V_{j+1}$ be inclusion maps. We say that the \emph{inductive limit} is the space $W$ with the greatest possible topology such that $\tilde{i}_j$ is continuous for all $j$. We write this as
    $$ W = \indlim_{j\geq 0} V_j. $$
\end{enumerate}
\end{definition}
With these definitions and propositions in mind, we can construct topologies on $S_s$, $S^\sigma$, $\Sigma_s$ and $\Sigma^\sigma$ and define their duals. For more information on topological vector spaces, see for instance \cite{Schaefer}.
\begin{definition}\label{def:top}
\begin{enumerate}
    \item Let $V_{s,r,N}(\mathbb{R}^n)$ consist of all $f\in C^\infty(\mathbb{R}^n)$ such that
 \begin{equation*}
     ||f||_{s,r,N}= \sup_{x\in\mathbb{R}^n, |\alpha|\leq N} \left|D^\alpha f(x)e^{r|x|^{1/s}}\right| < \infty.
 \end{equation*}
 \item Let $V_{r,M}^\sigma(\mathbb{R}^n)$ consist of all $f\in C^\infty(\mathbb{R}^n)$ such that
\begin{equation*}
     ||f||_{r,M}^\sigma= \sup_{\xi\in\mathbb{R}^n, |\beta|\leq M} \left| D^\beta \hat{f}(\xi)e^{r|\xi|^{1/\sigma}}\right| < \infty.
\end{equation*}
\end{enumerate}
\end{definition}
We see that
\begin{equation*}
    S_s(\mathbb{R}^n) = \indlim_{r>0}\left(\projlim_{N\geq 0} V_{s,r,N}(\mathbb{R}^n) \right)
\end{equation*}
and
\begin{equation*}
    S^\sigma(\mathbb{R}^n) = \indlim_{r>0}\left(\projlim_{M\geq 0} V_{r,M}^\sigma(\mathbb{R}^n) \right),
\end{equation*}
which implies
\begin{equation*}
    S_s(\mathbb{R}^n) = \bigcup_{r>0}\left(\bigcap_{N\geq 0} V_{s,r,N}(\mathbb{R}^n) \right),\quad S^\sigma(\mathbb{R}^n) = \bigcup_{r>0}\left(\bigcap_{M\geq 0} V_{r,M}^\sigma(\mathbb{R}^n) \right).
\end{equation*}
For the $\Sigma_s$- and $\Sigma^\sigma$-spaces, we obtain
\begin{equation*}
    \Sigma_s(\mathbb{R}^n) = \projlim_{r>0}\left(\projlim_{N\geq 0} V_{s,r,N}(\mathbb{R}^n) \right)
\end{equation*}
and
\begin{equation*}
    \Sigma^\sigma(\mathbb{R}^n) = \projlim_{r>0}\left(\projlim_{M\geq 0} V_{r,M}^\sigma(\mathbb{R}^n) \right)
\end{equation*}
which implies
\begin{equation*}
    \Sigma_s(\mathbb{R}^n) = \bigcap_{r>0}\left(\bigcap_{N\geq 0} V_{s,r,N}(\mathbb{R}^n) \right),\quad  \Sigma^\sigma(\mathbb{R}^n) = \bigcap_{r>0}\left(\bigcap_{M\geq 0} V_{r,M}^\sigma(\mathbb{R}^n) \right).
\end{equation*}
\begin{remark}
While $\Sigma_s$ and $\Sigma^\sigma$ are Fréchet spaces for all $s,\sigma>0$, the same is not known to be true for $S_s$ and $S^\sigma$ in current literature. 
\end{remark}
This leads us to define the dual spaces of $S_s$ and $S^\sigma$ in the following way.
\begin{definition}\label{def:Sdual}
We will denote a functional $f$ of such a dual space being applied to a test function $\psi$ in the appropriate corresponding space by $f(\psi) = \langle f, \psi \rangle$.
\begin{enumerate}[label=(\roman*)]
    \item We say that $u\in (S_s)'(\mathbb{R}^n)$ if for every $r>0$ there is an $N\geq 0$ such that
    \begin{equation*}
        |\langle u, f \rangle | \leq C_N \sum_{|\alpha|\leq N} || D^\alpha f e^{r|\cdot|^{1/s}} ||_\infty,
    \end{equation*}
    for any $f\in S_s(\mathbb{R}^n)$.
     \item We say that $u\in (S^\sigma)'(\mathbb{R}^n)$ if for every $r>0$ there is an $N\geq 0$ such that
    \begin{equation*}
        |\langle u, f \rangle | \leq C_N \sum_{|\alpha|\leq N} || D^\alpha \hat{f} e^{r|\cdot|^{1/\sigma}} ||_\infty,
    \end{equation*}
    for any $f\in S^\sigma(\mathbb{R}^n)$.
\end{enumerate}
\end{definition}
Similarly, we define the dual spaces of $\Sigma_s$ and $\Sigma^\sigma$ as follows.
\begin{definition}\label{def:Sigdual}
\begin{enumerate}[label=(\roman*)]
    \item We say that $u\in (\Sigma_s)'(\mathbb{R}^n)$ if there is an $r_0>0$ and an $N\geq 0$ such that
    \begin{equation*}
        |\langle u, f \rangle | \leq C \sum_{|\alpha|\leq N} || D^\alpha f e^{r_0|\cdot|^{1/s}} ||_\infty,
    \end{equation*}
    for any $f\in \Sigma_s(\mathbb{R}^n)$.
     \item We say that $u\in (\Sigma^\sigma)'(\mathbb{R}^n)$ if there is an $r_0>0$ and an $N\geq 0$ such that
    \begin{equation*}
        |\langle u, f \rangle | \leq C_N \sum_{|\alpha|\leq N} || D^\alpha \hat{f} e^{r_0|\cdot|^{1/\sigma}} ||_\infty,
    \end{equation*}
    for any $f\in \Sigma^\sigma(\mathbb{R}^n)$.
\end{enumerate}
\end{definition}
\begin{remark}
Since $S_0(\mathbb{R}^n) = C_c^\infty(\mathbb{R}^n)$, the space of compactly supported smooth functions, (cf. \cite[p.~170]{gel}) we have $(S_0)'(\mathbb{R}^n) = \mathscr{D}'(\mathbb{R}^n)$. Since $S_s(\mathbb{R}^n)$ is continuously embedded and dense in $S_{s'}(\mathbb{R}^n)$ for $s\leq s'$, we thus have $$(S_s)'(\mathbb{R}^n) \subseteq \mathscr{D}'(\mathbb{R}^n)$$ for all positive $s$. By Proposition \ref{prop:FT}, we therefore have $$(S^\sigma)'(\mathbb{R}^n) \subseteq \mathscr{F}\mathscr{D}'(\mathbb{R}^n)$$
for all positive $\sigma$.
\end{remark}
In Definitions \ref{def:Sdual} and \ref{def:Sigdual}, we can replace the $L^\infty(\mathbb{R}^n)$-norm with the $L^2(\mathbb{R}^n)$-norm by the arguments of \cite[p.~134]{eijndhoven}. We can extend this further with Hölder's inequality to obtain the following equivalent definitions of the dual spaces.
\begin{prop}\label{prop:Ssdual}
 Suppose $1\leq p \leq \infty$.
\begin{enumerate}[label=(\roman*)]
    \item $u\in (S_s)'(\mathbb{R}^n)$ if and only if for every $r>0$ there is an $N\geq 0$ such that
    \begin{equation*}
        |\langle u, f \rangle | \leq C_N \sum_{|\alpha|\leq N} || D^\alpha f e^{r|\cdot|^{1/s}} ||_p,
    \end{equation*}
    for any $f\in S_s(\mathbb{R}^n)$.
     \item $u\in (S^\sigma)'(\mathbb{R}^n)$ if and only if for every $r>0$ there is an $N\geq 0$ such that
    \begin{equation*}
        |\langle u, f \rangle | \leq C_N \sum_{|\alpha|\leq N} || D^\alpha \hat{f} e^{r|\cdot|^{1/\sigma}} ||_p,
    \end{equation*}
    for any $f\in S^\sigma(\mathbb{R}^n)$.
\end{enumerate}
\end{prop}
Replacing the $L^\infty$-norm with $L^p$-norms, $1\leq p < \infty$, is possible for the $\Sigma_s$ and $\Sigma^\sigma$ duals by similar arguments.

For $f\in (S_s)'$ or $f\in (S^\sigma)'$ we will also consider
$(f,\psi)$, which we denote to mean the continuous extension of the regular inner product of $L^2$ given by
\begin{equation*}
    (f,\psi)_2 = \int_{\mathbb{R}^n} f(y) \overline{\phi(y)} \, d y
\end{equation*}
to $f\in(S_s)'$, $\psi\in S_s$ or $f\in(S^\sigma)'$, $\psi\in S^\sigma$. The fact that this inner product can be extended continuously to duals of Gelfand-Shilov spaces follows by the fact that $(S_s,L^2,(S_s)')$ forms a \emph{Gelfand triple} (cf. \cite{bannert}). The same is true of $S^\sigma$, $\Sigma_s$ and $\Sigma^\sigma$.

We also recall the following definition of the short-time Fourier transform, which serves a pivotal role in several of the characterizations in this paper.
\begin{definition}
 The \emph{short-time Fourier transform} of $f\in(S_s)'(\mathbb{R}^n)$ with window function $\phi\in S_s(\mathbb{R}^n)$ is given by
\begin{equation*}
    V_\phi f(x,\xi) = \mathscr{F}\Big[ f \overline{\phi(\cdot - x)} \Big] (\xi) = (2\pi)^{-n/2}(f, \phi(\cdot-x)e^{i\langle \cdot,\xi\rangle}).
\end{equation*}

For $f$ belonging to $(S^\sigma)'(\mathbb{R}^n)$, $(\Sigma_s)'(\mathbb{R}^n)$ or $(\Sigma^\sigma)'(\mathbb{R}^n)$, we define the short-time Fourier transform by replacing each occurrence of $S_s$ above with $S^\sigma$, $\Sigma_s$ and $\Sigma^\sigma$, respectively.
\end{definition}

\section{Non-triviality of
$\Sigma_s^\sigma$-spaces}\label{sec:nontriv}
In this section, we determine when $\Sigma_s^\sigma$-spaces are nontrivial. Similar results have already been established for $S_s^\sigma$-spaces (cf. \cite{gel}). By nontrivial we mean that the space contains a function which is not constantly equal to zero. To establish non-triviality conditions, we will need two propositions.

The following proposition follows by similar arguments to those in \cite[p.~172-175]{gel}.
\begin{prop}\label{prop:SigCest}
If $s,\sigma >0$, $\sigma < 1$ and $f\in \Sigma_s^\sigma(\mathbb{R}^n)$, then $f$ can be continued analytically as an entire function in the ($2n$-dimensional) complex plane. Moreover, for every $a,b>0$,
\begin{equation*}
    |f(x+i y)| \leq C \exp{\left(-a|x|^{1/s} + b |y|^{1/(1-\sigma)}\right)}
\end{equation*}
for some constant $C=C_{a,b}$.
\end{prop}
We also find the following result in \cite[p.~228-233]{gel}.
\begin{prop}\label{prop:Striv}
For positive $s$ and $\sigma$, the space $S_s^\sigma(\mathbb{R}^n)$ is nontrivial if and only if $s+\sigma \geq 1$.
\end{prop}
With these propositions in mind, we prove the main result of this section. In previous works the condition ``$s+\sigma \geq 1$, $(s,\sigma) \neq (1/2,1/2)$'' is employed instead of ``$s+\sigma > 1$''. Here, we prove that the correct condition is $s+\sigma > 1$.
\begin{thm} \label{thm:Sigtriv}
Suppose $s,\sigma >0$. Then the space $\Sigma_s^\sigma(\mathbb{R}^n)$ is nontrivial if and only if $s+\sigma > 1$.
\end{thm}
\begin{proof}
Since $\Sigma_s^\sigma \subseteq S_s^\sigma$, it follows by Proposition \ref{prop:Striv} that $\Sigma_s^\sigma$ is trivial whenever $s+\sigma < 1$. Furthermore, Proposition \ref{prop:Striv} with Proposition \ref{prop:SsubSig} implies that $\Sigma_s^\sigma$ is nontrivial when $s+\sigma > 1$. Thus we need only consider the case $s+\sigma = 1$. Since $s$ and $\sigma$ are both assumed to be positive, we must have $\sigma < 1$. By Proposition \ref{prop:SigCest}, it is then true that
\begin{equation*}
    |f(z)| \leq C_{a,b} \exp{\left(-a|x|^{1/s} + b |y|^{1/s}\right)}
\end{equation*}
for every $a,b>0$, where $z=x+i y$. Moreover
\begin{equation*}
    |f(i z)| \leq C_{a,b} \exp{\left(-a|y|^{1/s} + b |x|^{1/s}\right)}
\end{equation*}
and therefore
\begin{equation*}
    | f(z)\cdot f(i z) | \leq C_{a,b}^2 \exp{\left( (b-a)( |x|^{1/s} + |y|^{1/s}) \right)}.
\end{equation*}
Since this inequality holds for all $a,b>0$, then by picking $a>b$ we see that $g(z) = f(z)\cdot f(i z)$ is bounded and tends to zero as $|x|,|y|\rightarrow \infty$. By Proposition \ref{prop:SigCest}, $f$ is an entire function, thus so is $g$. Hence Liouville's theorem implies that $g \equiv 0$. But this implies that $f\equiv 0$ as well, completing the proof.
\end{proof}
\section{Characterizations by short-time Fourier transform}\label{sec:stft}
We now move on to the characterization of $S_s$- and $S^\sigma$-spaces in terms of their short-time Fourier transforms. This is detailed in the following theorem, which is the main result of this section.
\begin{thm}\label{thm:SsSTFT}
Suppose $s,\sigma> 0$.
\begin{enumerate}[label=(\roman*)]
    \item Let $\phi\in S_s(\mathbb{R}^n)\setminus\{0\}$. Then $f\in S_s(\mathbb{R}^n)$ if and only if there is an $r>0$ such that 
        \begin{equation}\label{eq:STFT3}
            |V_\phi f(x,\xi) | \leq C_N (1+|\xi|^2)^{-N} e^{- r |x|^{1/s}}
        \end{equation}
        for every $N\geq 0$.
     \item Let $\phi\in S^\sigma(\mathbb{R}^n)\setminus\{0\}$. Then $f\in S^\sigma(\mathbb{R}^n)$ if and only if there is an $r>0$ such that
        \begin{equation}\label{eq:STFT4}
            |V_{\phi} f(x,\xi) | \leq C_N (1+|x|^2)^{-N} e^{- r |\xi|^{1/\sigma}}
        \end{equation}
        for every $N\geq 0$.
\end{enumerate}
\end{thm}
For the proof, we will need the following three lemmas. These lemmas follow by basic computations and are left for the reader to prove.
\begin{lemma}\label{lem:twistconv}
For $f\in S_s(\mathbb{R}^n)$ ($f\in\Sigma_s(\mathbb{R}^n)$) and $\phi_1,\phi_2,\phi_3 \in S_s(\mathbb{R}^n)$ ($\phi_1,\phi_2,\phi_3 \in \Sigma_s(\mathbb{R}^n)$),
\begin{equation*}
    (\phi_3,\phi_1) V_{\phi_2} f(x,\xi) = \frac{1}{(2\pi)^{n/2}} \iint V_{\phi_1}f(x-y,\xi-\eta) V_{\phi_2} \phi_3 (y,\eta) e^{-i\langle x-y,\eta \rangle} \, d y \, d \eta.
\end{equation*}
\end{lemma}
\begin{lemma}\label{lem:absolutineq}
If $s>0$ then there is a $C \geq 1$ such that
\begin{equation}\label{eq:absolutineq}
    C^{-1} (|x|^{1/s} + |y|^{1/s} ) \leq |y|^{1/s} + |y-x|^{1/s} \leq C (|x|^{1/s} + |y|^{1/s} ).
\end{equation}
for every $x,y\in\mathbb{R}^n$.
\end{lemma}
\begin{lemma}\label{lem:absolutineq3}
For any $\xi,\eta\in \mathbb{R}^n$ and any $N\geq 0$, there is a constant $C>0$ such that
\begin{equation*}
    (1+|\xi-\eta|^2)^{-N} \leq C (1+|\xi|^2)^{-N} (1+|\eta|^2)^N.
\end{equation*}
\end{lemma}
\begin{proof}[Proof of Theorem \ref{thm:SsSTFT}]
Suppose that $f\in S_s$. For every $N\geq 0$, we have
\begin{align*}
    |(1+|\xi|^2)^N V_\phi f(x,\xi) | &= \frac{1}{(2\pi)^{n/2}} \left| \int_{\mathbb{R}^n} f(y) \overline{\phi(y-x)} (1+|\xi|^2)^N e^{-i\langle y,\xi \rangle} \, d y \right| \\
    &=\frac{1}{(2\pi)^{n/2}} \left| \int_{\mathbb{R}^n} f(y) \overline{\phi(y-x)} (1-\Delta)^N e^{-i\langle y,\xi \rangle} \, d y \right| \\
    &= \frac{1}{(2\pi)^{n/2}} \left| \sum_{\gamma_0+|\gamma|=N} \dfrac{N!}{\gamma_0! \gamma!} \int_{\mathbb{R}^n} f(y) \overline{\phi(y-x)} D^{2\gamma} e^{-i\langle y, \xi \rangle} \, d y, \right|
\end{align*}
where $\gamma'=(\gamma_0,\gamma)\in \mathbb{N}^{1+n}$ and the derivatives are taken with respect to $y$. Integration by parts together with Leibniz formula yields
\begin{align*}
    |(1+|\xi|^2)^N V_\phi f(x,\xi) | &= \frac{1}{(2\pi)^{n/2}} \left|\sum_{\gamma_0+|\gamma|=N} \dfrac{N!}{\gamma_0! \gamma!} \int_{\mathbb{R}^n} D^{2\gamma} \left[ f(y) \overline{\phi(y-x)} \right] e^{-i \langle y, \xi \rangle}\, d y \right|
    \\ &=\frac{1}{(2\pi)^{n/2}} \left|\sum_{\gamma',\alpha}  c^N_{\gamma',\alpha} \int D^{\alpha} f(y) D^{2\gamma-\alpha} \overline{\phi(y-x)} e^{-i\langle y, \xi\rangle} d y \right| \\
    &\leq \frac{1}{(2\pi)^{n/2}} \sum_{\gamma',\alpha} c^N_{\gamma',\alpha} \int\left| D^{\alpha} f(y) D^{2\gamma-\alpha} \overline{\phi(y-x)} \right| \, d y,
\end{align*}
where $\sum_{\gamma',\alpha}=\sum_{|\gamma'|=N} \sum_{\alpha\leq \gamma}$ and where $c^N_{\gamma',\alpha} = \dfrac{N!}{\gamma_0! \gamma!}\binom{2\gamma}{\alpha}$. By Proposition \ref{prop:Ssalt} there are $C_{\gamma,\alpha}, a>0$ such that
\begin{align*}
    \int\left| D^{\alpha} f(y) D^{2\gamma-\alpha} \overline{\phi(y-x)} \right| \, d y &\leq C_{\gamma,\alpha} \int e^{- a(|y|^{1/s}+|y-x|^{1/s})} \, d y,
\end{align*}
and by Lemma \ref{lem:absolutineq} there is a $c>0$ such that
\begin{equation*}
    \int\left| D^{\alpha} f(y) D^{2\gamma-\alpha} \overline{\phi(y-x)} \right| \, d y \leq C'_{\gamma,\alpha} e^{-ac|x|^{1/s}}
\end{equation*}
since $\int e^{-ac|y|^{1/s}} d y < \infty$. Hence, with $r=a c>0$ and $C_{N,\gamma',\alpha} = c^N_{\gamma',\alpha}C'_{\gamma,\alpha}>0$ we obtain
\begin{align*}
    |(1+|\xi|^2)^N V_\phi f(x,\xi) | &\leq \frac{1}{(2\pi)^{n/2}} \sum_{\gamma',\alpha}  C_{N,\gamma',\alpha} e^{-r|x|^{1/s}} \\
    &\leq C_N e^{-r|x|^{1/s}},
\end{align*}
where $C_N = \frac{1}{(2\pi)^{n/2}}\sum_{\gamma',\alpha} C_{N,\gamma',\alpha}$. Thus \eqref{eq:STFT3} holds for every $N\geq 0$.

Now suppose that \eqref{eq:STFT3} holds for every $N\geq 0$. This condition implies that $f\in \mathscr{S}$ (cf. \cite{grochenig}). In particular $f\in C^\infty$, hence by Proposition \ref{prop:Ssalt}, the result follows if there is an $r>0$ such that \eqref{eq:Ssalt} holds for every multi-index $\beta$.

Consider $V_\phi[D^\beta f](x,\xi)$. Integrating by parts and applying Leibniz formula gives
\begin{align*}
    V_\phi[D^\beta f](x,\xi) &= \frac{(-1)^\beta}{(2\pi)^{n/2}}\int f(y) D^\beta \left( \overline{\phi(y-x)} e^{-i\langle y,\xi \rangle} \right) \, d y \\
    &= \sum_{\alpha\leq\beta} \frac{C_{\alpha,\beta}}{(2\pi)^{n/2}} \int f(y) D^{\alpha}\overline{\phi(y-x)} \xi^{\beta-\alpha} e^{-i\langle y,\xi \rangle}\, d y \\
    &= \sum_{\alpha\leq\beta} C_{\alpha,\beta} \xi^{\beta-\alpha} V_{D^{\alpha}\phi} f(x,\xi),
\end{align*}
where $C_{\alpha,\beta} = (-1)^\beta \binom{\beta}{\alpha}$.
By Lemma \ref{lem:twistconv} we therefore have
\begin{equation}\label{eq:twist1}
    V_\phi[D^\beta f](x,\xi) =\sum_{\alpha\leq\beta} C'_{\alpha,\beta} \xi^{\beta-\alpha} \iint V_\phi f(x-y,\xi-\eta) V_{D^\beta\phi}\phi(y,\eta) e^{-i\langle x-y,\eta \rangle} d y \, d \eta,
\end{equation}
where $C'_{\alpha,\beta}=(2\pi)^{-n/2} (\phi,\phi)^{-1}C_{\alpha,\beta}$. 

For now, we consider only the double integral 
\begin{equation*}
    I = \iint V_\phi f(x-y,\xi-\eta) V_{D^\beta\phi}\phi(y,\eta) e^{-i\langle x-y,\eta \rangle} d y \, d \eta
\end{equation*}
from the right-hand side of the previous equation. Note that
\begin{equation*}
    V_{D^\beta \phi}\phi(x,\xi) = e^{-i\langle x,\xi\rangle} V_{\overline{\phi}} [D^\beta\overline{\phi}](-x,\xi),
\end{equation*}
and since $D^\beta\overline{\phi},\overline{\phi}\in S_s\setminus\{0\}$ the first part of this theorem now implies that there is an $r_1 > 0$ such that
\begin{equation*}
    |V_{D^\beta \phi}\phi(x,\xi)| \leq C_{\beta,N_1} (1+|\xi|^2)^{-N_1} e^{-r_1 |x|^{1/s}}
\end{equation*}
for every $N_1\geq 0$ and every $\beta$. (Note that all the derivatives of $\overline{\phi}$ fulfill Proposition \ref{prop:Ssalt} with the same exponent, hence we can use the same $r_1>0$ for every $\beta$.) By assumption, \eqref{eq:STFT3} holds for all $N\geq 0$. For any given $N$, pick $N_1 > N$. We now obtain
\begin{align*}
    I &\leq A_{\beta,N} \iint (1+|\xi-\eta|^2)^{-N} e^{- r |x-y|^{1/s}}(1+|\eta|^2)^{-N_1} e^{-r_1 |y|^{1/s}} d y d \eta \\
    &= A_{\beta,N} I_1 I_2
\end{align*}
where $A_{\beta,N} = C_N C_{\beta,N_1}$, $$I_1 = \int e^{- r |x-y|^{1/s}} e^{-r_1 |y|^{1/s}} d y$$ and $$I_2 = \int (1+|\xi-\eta|^2)^{-N}(1+|\eta|^2)^{-N_1} d \eta.$$  
In order to estimate $I_1$, we let $r_2 = \min\{r,r_1\}$ and apply Lemma \ref{lem:absolutineq} to obtain $c>0$ such that
\begin{align*}
    I_1 &\leq \int e^{-r_2( |y-x|^{1/s} + |y|^{1/s})} d y \\
    &\leq e^{- r_2 c |x|^{1/s}} \int e^{- r_2 c |y|^{1/s}} d y \\
    &= B e^{- r_2 c |x|^{1/s}},
\end{align*}
where $B = \int e^{- r_2 c |y|^{1/s}} d y < \infty$. Since $N_1 > N$, Lemma \ref{lem:absolutineq3} gives
\begin{align*}
    I_2 &= \int (1+|\xi-\eta|^2)^{-N}(1+|\eta|^2)^{-N_1} d \eta \\
    &\leq C (1+|\xi|^2)^{-N}\int (1+|\eta|^2)^{N - N_1} d \eta \\
    &= B_N (1 + |\xi|^2)^{- N},
\end{align*}
where $B_N = C \int (1+|\eta|^2)^{N - N_1} d \eta < \infty$.
Combining these estimates, we get
\begin{equation*}
    I \leq B_{\beta,N} (1+|\xi|^2)^{-N} e^{- r_2 c |x|^{1/s}}
\end{equation*}
for every $N\geq 0$, where $B_{\beta,N} = A_{\beta,N} B B_N $.  Combining this with \eqref{eq:twist1}, we obtain
\begin{equation} \label{eq:twist2}
    | V_\phi [D^\beta f](x,\xi) | \leq \sum_{\alpha\leq\beta} B_{\alpha,\beta,N} \xi^{\beta - \alpha} (1+|\xi|^2)^{-N} e^{- c r |x|^{1/s}}
\end{equation}
for every $N\geq 0$, where $B_{\alpha,\beta,N} = C'_{\alpha,\beta} B_{\beta,N}$. 

We now integrate both sides of \eqref{eq:twist2} with respect to $\xi$. Note that
\begin{equation*}
    |V_\phi f(x,\xi)| = (2\pi)^{-n/2}\left|\left(\hat{f}_{-x}*\psi\right) (\xi)\right|,
\end{equation*}
where $\psi = \mathscr{F}\left[\,\overline{\phi}\,\right]$, and since $\mathscr{F}[f_a](\eta) = e^{-i\langle a, \eta \rangle}\hat{f}(\eta)$,
\begin{align*}
    (2\pi)^{-n/2}  \int_{\mathbb{R}^n}\left|\left(\hat{f}_{-x}*\psi\right) (\xi) \right|\, d \xi &\geq (2\pi)^{-n/2} \left| \int_{\mathbb{R}^n}\left(\hat{f}_{-x}*\psi\right) (\xi) \, d \xi \right| \\
    &= (2\pi)^{-n/2} \left|\iint \hat{f}_{-x}(\eta)\psi(\xi - \eta) \, d \eta \, d \xi  \right| \\
    &= (2\pi)^{-n/2} \left| \int e^{i\langle x, \eta \rangle}\hat{f}(\eta) \, d \eta \int \psi(\xi - \eta) \, d \xi \right| \\
    &= |f(x)|\left|\int \psi(\xi - \eta) \, d \xi\right|.
\end{align*}
Since $\int \psi(\xi - \eta) \, d \xi <\infty$, we therefore obtain
\begin{equation}\label{eq:twist3}
    |D^\beta f(x)| \leq C_\phi \int | V_\phi [D^\beta f](x,\xi) | d \xi,
\end{equation}
for some constant $C_\phi > 0$. Moreover, if we fix $N > |\beta|$, then
\begin{equation*}
    \int \xi^{\beta - \alpha} (1+|\xi|^2)^{-N} d \xi = D_{\alpha,\beta} < \infty 
\end{equation*}
for each $\alpha \leq \beta$ and thus, with $r' = r_2 c$,
\begin{equation} \label{eq:twist4}
    \int | V_\phi [D^\beta f](x,\xi) | d \xi \leq \sum_{\alpha\leq\beta} B_{\alpha,\beta,N} D_{\alpha,\beta} e^{- r' |x|^{1/s}}.
\end{equation}
Finally let $C_\beta = C_\phi^{-1}\sum_{\alpha\leq\beta} B_{\alpha,\beta,N} D_{\alpha,\beta}$. Then combining \eqref{eq:twist3} with \eqref{eq:twist4} now yields
\begin{equation*}
    |D^\beta f(x) | \leq C_\beta e^{- r' |x|^{1/s}}
\end{equation*}
for every multi-index $\beta$. This completes the proof of (i). 

To prove (ii), we first note that by Proposition \ref{prop:FT}, $f\in S^\sigma$ and $\phi\in S^\sigma\setminus\{0\}$ if and only if $\hat{f}\in S_\sigma$ and $\hat{\phi}\in S_\sigma\setminus\{0\}$. By (i), we therefore have $f\in S^\sigma$ if and only if
\begin{equation*}
    |V_{\hat{\phi}} \hat{f}(x,\xi)| \leq C_N (1+|\xi|^2)^{-N} e^{-r |x|^{1/\sigma}}
\end{equation*}
for every $N\geq 0$. Since $V_{\hat{\phi}} \hat{f}(x,\xi) = e^{-i\langle x,\xi \rangle} V_{\phi}f (-\xi,x)$, this condition can be rewritten as
\begin{equation*}
    |V_{\phi} f(-\xi,x)| \leq C_N (1+|\xi|^2)^{-N} e^{-r |x|^{1/\sigma}}.
\end{equation*}
Performing a variable substitution now yields \eqref{eq:STFT4}. This completes the proof.
\end{proof}

Utilizing Proposition \ref{prop:Ssalt}(b) instead of Proposition \ref{prop:Ssalt}(a), we obtain the following characterizations of short-time Fourier transforms in $\Sigma_s$ and $\Sigma^\sigma$ by analogous arguments.
\begin{thm}
Suppose $s,\sigma> 0$.
\begin{enumerate}[label=(\roman*)]
    \item Let $\phi\in \Sigma_s(\mathbb{R}^n)\setminus\{0\}$. Then $f\in \Sigma_s(\mathbb{R}^n)$ if and only if for every $r>0$ and every $N\geq 0$, 
        \begin{equation*}
            |V_\phi f(x,\xi) | \leq C_{r,N} (1+|\xi|^2)^{-N} e^{- r |x|^{1/s}}.
        \end{equation*}
    \item Let $\phi\in \Sigma^\sigma(\mathbb{R}^n)\setminus\{0\}$. Then $f\in \Sigma^\sigma(\mathbb{R}^n)$ if and only if for every $r>0$ and every $N\geq 0$,
\begin{equation*}
    |V_\phi f(x,\xi) | \leq C_{r,N} (1+|x|^2)^{-N} e^{- r |\xi|^{1/\sigma}}.
\end{equation*}
\end{enumerate}
\end{thm}

Using these short-time Fourier transform characterizations, we can obtain characterizations for the one-parameter spaces similar to those of \cite{chung} for the two-parameter spaces.
\begin{thm}\label{thm:SsFT}
Suppose $s,\sigma>0$ and $f\in C^\infty(\mathbb{R}^n)$.
\begin{enumerate}[label=(\alph*)]

    \item $f\in S_s(\mathbb{R}^n)$ if and only if there is an $r>0$ such that
    \begin{equation}\label{eq:Ssaltchar}
        |f(x)|\leq C e^{-r|x|^{1/s}}, \quad |\hat{f}(\xi)|\leq C_N (1+|\xi|^2)^{-N} 
    \end{equation}
    for every $N\geq 0$.
    \item $f\in S^\sigma(\mathbb{R}^n)$ if and only if there is an $r>0$ such that
    \begin{equation}\label{eq:Ssigaltchar}
        |f(x)|\leq C_N (1+|x|^2)^{-N} , \quad |\hat{f}(\xi)|\leq C e^{-r|\xi|^{1/\sigma}}
    \end{equation}
    for every $N\geq 0$.
\end{enumerate}
\end{thm}
\begin{proof}
Suppose first that $f\in S_s$. Then there is an $r>0$ such that \eqref{eq:STFT3} holds for every $N\geq 0$.

Integrating both sides of this equation with respect to $\xi$ yields
\begin{equation*}
    \int |V_\phi f(x,\xi)| \, d \xi \leq C' e^{-r|x|^{1/s}}.
\end{equation*}
On the other hand, since $V_\phi f = (\hat{f},\hat{\phi}(\cdot-\xi)e^{-i\langle \cdot, x \rangle}),$
\begin{align*}
    \int |V_\phi f(x,\xi)| \, d \xi &\geq (2\pi)^{-n/2}\left| \int e^{i\langle x,\eta\rangle} \hat{f}(\eta)  \int \overline{\hat{\phi}(\xi-\eta)}\, d \xi \, d \eta \right| \\
    &= C_\phi |f(x)|,
\end{align*}
hence we obtain $$ |f(x)| \leq C e^{-r|x|^{1/s}}.$$

Starting instead by integrating both sides of \eqref{eq:STFT3} with respect to $x$ yields the inequalities
\begin{align*}
    C_\phi'' |\hat{f}(\xi)| &\leq (2\pi)^{-n/2} \int f(y) e^{-\langle y,\xi \rangle} \int \overline{\phi(y-x)} \, d x \, d y  \\ &\leq \int |V_\phi f(x,\xi) |\, d x \\
    &\leq C'_N (1+|\xi|^2)^{-N}
\end{align*}
hence we obtain $$ |\hat{f}(\xi)|\leq C_N (1+|\xi|^2)^{-N}$$
for every $N\geq 0$.

Suppose instead that $f$ fulfills \eqref{eq:Ssaltchar} for some $r>0$ and every $N\geq 0$. Then
$$ |V_\phi (x,\xi)| \leq (2\pi)^{-n/2}\int |f(y)| |\phi(y-x)| \, d y. $$
Using what we proved in the first half of this proof, there are $C_0,r_1,r_2 > 0$ such that
$$ \int |f(y)| |\phi(y-x)| \, d y \leq C_0 \int e^{-r_1 |y|^{1/s}}e^{-r_2|y-x|^{1/s}} \, d y \leq C_1 e^{-r|x|^{1/s}} $$
for some $r>0$, where we use Lemma \ref{lem:absolutineq} for the last inequality. Hence
$$ |V_\phi f(x,\xi)| \leq C e^{-r |x|^{1/s}}. $$
Using the same strategy once more but starting with the fact that
$$ |V_\phi f (x,\xi) | \leq (2\pi)^{-n/2} \int |\hat{f}(\eta)| |\hat{\phi}(\eta - \xi)| \, d \eta, $$
and this time utilizing Lemma \ref{lem:absolutineq3} instead, we obtain
$$ |V_\phi f(x,\xi) | \leq C_N (1+|\xi|^2)^{-N} $$
for every $N\geq 0$. Combining both of these inequalities we see that for every $N\geq 0$, 
$$ |V_\phi f(x,\xi)|^2 \leq C_N (1+|\xi|^2)^{-N} e^{- r |x|^{1/s}}, $$
and in particular for $N=2k$, $k\geq 0$,
$$ |V_\phi f(x,\xi)|^2 \leq C_{2k} (1+|\xi|^2)^{-2k} e^{- r |x|^{1/s}}, $$
thus
$$ |V_\phi f(x,\xi) | \leq C_k' (1+|\xi|^2)^{-k} e^{-r' |x|^{1/s}} $$
for all $k\geq 0$, where $C_k' = \sqrt{C_{2k}}$ and $r'=r/2$. This completes the proof of (a).

To prove (b), simply perform Fourier transforms in light of Proposition \ref{prop:FT} and apply (a).
\end{proof}
As with the other results, we state the corresponding theorem for the $\Sigma_s$- and $\Sigma^\sigma$-spaces but omit its proof as it follows by analogous arguments.
\begin{thm}\label{thm:SigsFT}
Suppose $s,\sigma>0$ and $f\in C^\infty(\mathbb{R}^n)$.
\begin{enumerate}[label=(\alph*)]

    \item $f\in \Sigma_s(\mathbb{R}^n)$ if and only if for every $r>0$ and $N\geq 0$,
    \begin{equation}\label{eq:Sigmasaltchar}
        |f(x)|\leq C_r e^{-r|x|^{1/s}}, \quad |\hat{f}(\xi)|\leq C_N (1+|\xi|^2)^{-N}. 
    \end{equation}
    \item $f\in \Sigma^\sigma(\mathbb{R}^n)$ if and only if for every $r>0$ and every $N\geq 0$,
    \begin{equation}\label{eq:Sigmasigaltchar}
        |f(x)|\leq C_N (1+|x|^2)^{-N} , \quad |\hat{f}(\xi)|\leq C_r e^{-r|\xi|^{1/\sigma}}.
    \end{equation}
\end{enumerate}
\end{thm}
\section{Characterizations of dual spaces}\label{sec:dual}
We now move on to the characterization of duals to one-parameter spaces $S_s$ and $S^\sigma$. These duals were defined in Section \ref{sec:prelim} via the topologies detailed in Definition \ref{def:top} and using the results of Section \ref{sec:stft}, we now arrive at the following equivalent topologies. 
\begin{prop}
Suppose $s,\sigma > 0$.
\begin{enumerate}[label=(\roman*)]
    \item Let $\phi\in S_s(\mathbb{R}^n)\setminus\{0\}$, let
    \begin{equation*}
        p^{\phi}_{s,r,N}(f) = \sup_{x,\xi\in\mathbb{R}^n} \left| V_\phi f(x,\xi) (1+|\xi|^2)^{N}e^{r |x|^{1/s}}\right|
    \end{equation*}
    and let $B_{s,r,N}(\mathbb{R}^{n})$ be the Banach space consisting of all $f\in C^\infty(\mathbb{R}^{n})$ such that $p^{\phi}_{s,r,N}(f)$ is finite. Then 
    $$ S_s(\mathbb{R}^n) = \indlim_{r>0}\left(\projlim_{N\geq 0} B_{s,r,N}(\mathbb{R}^{n}) \right) $$
    where the equality holds in a topological sense as well. 
    \item Let $\phi\in S^\sigma(\mathbb{R}^n)\setminus\{0\}$, let
    \begin{equation*}
        q^{\phi,\sigma}_{r,M}(f) = \sup_{x,\xi\in\mathbb{R}^n} \left| V_\phi f(x,\xi) (1+|x|^2)^{M}e^{r |\xi|^{1/\sigma}}\right|
    \end{equation*}
    and let $B^\sigma_{r,M}(\mathbb{R}^{n})$ be the Banach space consisting of all $f\in C^\infty(\mathbb{R}^{n})$ such that $q^{\phi,\sigma}_{r,M}(f)$ is finite. Then
    $$ S^\sigma(\mathbb{R}^n) = \indlim_{r>0}\left(\projlim_{M\geq 0} B^\sigma_{r,M}(\mathbb{R}^{n}) \right) $$
    where the equality holds in a topological sense as well. 
\end{enumerate}
\end{prop}
\begin{proof}
The equivalence of the semi-norms $p^{\phi}_{s,r,N}$ and $||\cdot||_{s,r,N}$, as well as that of $q^{\phi,\sigma}_{r,M}$ and $||\cdot||^\sigma_{r,M}$ is established implicitly in the proof of Theorem \ref{thm:SsSTFT}.
\end{proof}
This proposition gives us the following equivalent definitions for the dual spaces $(S_s)'$ and $(S^\sigma)'$.
\begin{cor}
Suppose $s,\sigma > 0$.
\begin{enumerate}[label=(\roman*)]
    \item Let $\phi\in S_s(\mathbb{R}^n)\setminus\{0\}$ and let $u\in \mathscr{D}'(\mathbb{R}^n)$. Then $u\in (S_s)'(\mathbb{R}^n)$ if and only if for every $r>0$ there is an $N\geq 0$ such that
    \begin{equation*}
        |\langle u, f \rangle|\leq C_N p^{\phi}_{s,r,N}(f)
    \end{equation*}
    for all $f\in S_s(\mathbb{R}^n)$.
    \item Let $\phi\in S^\sigma(\mathbb{R}^n)\setminus\{0\}$ and $u\in\mathscr{F}\mathscr{D}'(\mathbb{R}^n)$. Then $u\in (S^\sigma)'(\mathbb{R}^n)$ if and only if for every $r>0$ there is an $M\geq 0$ such that
    \begin{equation*}
        |\langle u, f \rangle|\leq C_M q^{\phi,\sigma}_{r,M}(f)
    \end{equation*}
    for all $f\in S^\sigma (\mathbb{R}^n)$.
\end{enumerate}
\end{cor}
This brings us to the following characterization of the duals via short-time Fourier transforms, which is the main result of this section.
\begin{thm}\label{thm:SsSTFTdual}
\begin{enumerate}[label=(\roman*)]
    \item Let $\phi\in S_s(\mathbb{R}^n)\setminus\{0\}$ and $f\in \mathscr{D}'(\mathbb{R}^n)$. Then $f\in(S_s)'(\mathbb{R}^n)$ if and only if for every $r>0$ there is an $N_0\geq 0$ such that
    \begin{equation}\label{eq:STFTdual1}
        |V_\phi f(x,\xi)| \leq C_{r}(1+|\xi|^2)^{N_0} e^{r|x|^{1/s}}.
    \end{equation}
     \item Let $\phi\in S^\sigma(\mathbb{R}^n)\setminus\{0\}$ and $f\in \mathscr{F}\mathscr{D}'(\mathbb{R}^n)$. Then $f\in(S^\sigma)'(\mathbb{R}^n)$ if and only if for every $r>0$ there is an $N_0\geq 0$ such that
    \begin{equation}\label{eq:STFTdual2}
        |V_\phi f(x,\xi)| \leq C_{r}(1+|x|^2)^{N_0} e^{r|\xi|^{1/\sigma}}.
    \end{equation}
\end{enumerate}
\end{thm}
\begin{proof}
Suppose $f\in (S_s)'$, $\phi\in S_s\setminus\{0\}$. Then by Proposition \ref{prop:Ssdual} and Leibniz formula, there is an $N_0>0$ such that
\begin{align*}
    |V_\phi f(x,\xi)| &= |(f,\phi(\cdot-x)e^{i\langle\cdot,\xi\rangle})|  \\
    &\leq C_{N_0} \sum_{|\alpha|\leq N_0} || D^\alpha \left(\phi(\cdot-x)e^{i\langle \cdot,\xi\rangle}\right) e^{r|\cdot|^{1/s}}||_2\\
    &= C_{N_0} \sum_{|\alpha|\leq N_0}\sum_{\gamma\leq\alpha}\binom{\alpha}{\gamma} || (D^\gamma \phi)(\cdot-x) \xi^{\alpha-\gamma} e^{i\langle \cdot,\xi\rangle} e^{r|\cdot|^{1/s}}||_2 \\
    &\leq C'_{N_0} (1+|\xi|^2)^{N_0} \sum_{|\alpha|\leq N_0}\sum_{\gamma\leq\alpha}\binom{\alpha}{\gamma} || (D^\gamma \phi)(\cdot-x) e^{r|\cdot|^{1/s}}||_2
\end{align*}
By Theorem \ref{prop:Ssalt}, there is an $r_0>0$ such that $|D^{\gamma}\phi(y-x)|\leq C_\gamma e^{-r_0|y-x|^{1/s}}$, hence
\begin{align*}
    |V_\phi f(x,\xi)| &\leq C'_{N_0} (1+|\xi|^2)^{N_0} \sum_{|\alpha|\leq N_0}\sum_{\gamma\leq\alpha}\binom{\alpha}{\gamma} || e^{-r_0|\cdot-x|^{1/s}} e^{r|\cdot|^{1/s}}||_2.
\end{align*}
By Lemma \ref{lem:absolutineq}, there is a $c\geq 1$ such that
\begin{equation*}
    - r_0|y-x|^{1/s} \leq   r_0 |x|^{1/s} - r_0/c \cdot |y|^{1/s}.
\end{equation*}
Let $r\in (0, r_0 /(2c))$. Then
\begin{equation*}
    - r_0|y-x|^{1/s} \leq - 2 c r|y-x|^{1/s} \leq  2 c r |x|^{1/s} - 2 r |y|^{1/s},
\end{equation*}
and
\begin{align*}
    |V_\phi f(x,\xi)| &\leq C'_{N_0} (1+|\xi|^2)^{N_0} \sum_{|\alpha|\leq N_0}\sum_{\gamma\leq\alpha}\binom{\alpha}{\gamma} || e^{2cr |x|^{1/s} - 2 r |\cdot|^{1/s} + r|\cdot|^{1/s}}||_2 \\
    &=C'_{N_0} (1+|\xi|^2)^{N_0} e^{2 c r |x|^{1/s}} \sum_{|\alpha|\leq N_0}\sum_{\gamma\leq\alpha}\binom{\alpha}{\gamma} ||  e^{-r|\cdot|^{1/s}}||_2  \\
    &\leq C''_{N_0} (1+|\xi|^2)^{N_0} e^{2 c r |x|^{1/s}}.
\end{align*}
Clearly, the inequality still holds if we let $r>r_0/(2c)$ (the right hand side only becomes larger), hence we have shown that the inequality is valid for all $r>0$, as was to be shown.

Now suppose that for every $r>0$ there is an $N_0 \geq 0$ such that \eqref{eq:STFTdual1} holds. Then by Moyal's identity, for every $\varphi \in S_s $
\begin{equation*}
    |(f,\varphi)|\leq ||\phi||_2^{-2} |(V_\phi f, V_\phi \varphi)|, 
\end{equation*}
hence
\begin{equation*}
    |(f,\varphi)|\leq ||\phi||_2^{-2} \int \int |V_\phi f(x,\xi)|\cdot |V_\phi \varphi(x,\xi)| \, d x \, d \xi.
\end{equation*}
By assumption combined with Theorem \ref{thm:SsSTFT}, for every $r,r_1>0$ and any $N_1\geq 0$ there is an $N_0\geq 0$ such that
\begin{align*}
    |(f,\varphi)|&\leq C_{N_0}||\phi||_2^{-2} \int \int  (1+|\xi|^2)^{N_0} e^{r|x|^{1/s}}  |V_\phi\varphi(x,\xi)| \, d x \, d \xi \\
    &= C_{N_0}||\phi||_2^{-2} \int \int  (1+|\xi|^2)^{(N_0-N_1)} e^{(r-r_1)|x|^{1/s}}  \left|V_\phi\varphi(x,\xi)(1+|\xi|^2)^{N_1} e^{r_1|x|^{1/s}}\right| \, d x \, d \xi \\
    &\leq C_{N_0}||\phi||_2^{-2} p^{\phi,s}_{r_1,N_1}(\varphi)\int \int  (1+|\xi|^2)^{-(N_1-N_0)} e^{-(r_1-r)|x|^{1/s}} \, d x \, d \xi.
\end{align*}
Pick $N_1$ such that $N_1>N_0$ and pick $r$ such that $r<r_1$. Then we obtain
\begin{align*}
    |(f,\varphi)|&\leq   C_{N_0,N_1,r}||\phi||_2^{-2} p^{\phi,s}_{r_1,N_1}(\varphi)
\end{align*}
for all $r_1>r$. By picking $r>0$ arbitrarily small, we thus obtain
\begin{align*}
    |(f,\varphi)|&\leq   C'_{N_1,r_1}\cdot p^{\phi,s}_{r_1,N_1}(\varphi)
\end{align*}
for all $r_1>0$. This completes the proof of (i). The proof of (ii) is very similar, utilizing the fact that $\phi\in S^\sigma$ is equivalent to $\hat{\phi}\in S_\sigma $, and the fact that $V_\phi f(x,\xi) = (\hat{f},\hat{\phi}(\cdot-\xi)e^{-i\langle \cdot,x \rangle})$.    
    
\end{proof}
Lastly we include the corresponding result for the dual spaces of $\Sigma_s$ and $\Sigma^\sigma$, which follows by analogous arguments.
\begin{thm}
\begin{enumerate}[label=(\roman*)]
    \item Let $\phi\in \Sigma_s(\mathbb{R}^n)\setminus\{0\}$. Then $f\in(\Sigma_s)'(\mathbb{R}^n)$ if and only if there is an $r_0>0$, a $C>0$ and an $N_0\geq 0$ such that
    \begin{equation*}
        |V_\phi f(x,\xi)| \leq C(1+|\xi|^2)^{N_0} e^{r_0|x|^{1/s}}.
    \end{equation*}
     \item Let $\phi\in \Sigma^\sigma(\mathbb{R}^n)\setminus\{0\}$. Then $f\in(\Sigma^\sigma)'(\mathbb{R}^n)$ if and only if for every $r_0>0$ there is an $N_0\geq 0$ such that
    \begin{equation*}
        |V_\phi f(x,\xi)| \leq C(1+|x|^2)^{N_0} e^{r_0|\xi|^{1/\sigma}}.
    \end{equation*}
\end{enumerate}
\end{thm}
\section{Continuity of Toeplitz operators}\label{sec:top}
We will now look at Toeplitz operators on one-parameter Gelfand-Shilov spaces. To analyze these, we will need to consider functions in $2n$ dimensions which belong to different  one-parameter Gelfand-Shilov spaces in different ($n$-dimensional) variables. To make sense of these, we begin by examining the spaces where each variable belongs to a two-parameter Gelfand-Shilov space. These spaces are defined as follows.
    \begin{definition}
     Suppose $s_1,s_2,\sigma_1,\sigma_2 > 0$. Then $S_{s_1,\sigma_2}^{\sigma_1,s_2}(\mathbb{R}^{2 n})$ consists of every $f\in C^{\infty}(\mathbb{R}^{2n})$ for which there is an $h>0$ such that
     \begin{equation*}
         \sup \dfrac{\left| x^{\alpha_1} \xi^{\alpha_2} D_x^{\beta_1} D_\xi^{\beta_2} f(x,\xi) \right|}{h^{|\alpha_1 + \alpha_2 + \beta_1 + \beta_2|} (\alpha_1!)^{s_1} (\alpha_2!)^{\sigma_2} (\beta_1!)^{\sigma_1} (\beta_2!)^{s_2}} < \infty 
     \end{equation*}
    for every $\alpha_1,\alpha_2,\beta_1,\beta_2\in \mathbb{N}^n,$ where the supremum is taken over $x,\xi\in\mathbb{R}^n$.
    \end{definition}
    We can interpret this as a space where functions belong to $S_{s_1}^{\sigma_1}(\mathbb{R}^{n})$ in the $x$-variable and $S_{\sigma_2}^{s_2}(\mathbb{R}^{n})$ in the $\xi$-variable. Note that $S_{s,s}^{\sigma,\sigma}(\mathbb{R}^{2 n}) = S_s^\sigma(\mathbb{R}^{2 n})$. With the notations $S_{s}^\infty = S_s$ and $S_\infty^\sigma = S^\sigma$, we can construct similar spaces where the functions belong to the one-parameter spaces in single variables instead. These are the spaces we will focus on in this section.
    
    \begin{definition}
     Suppose $s,t> 0$. Then $S_{s,\infty}^{\infty,t}(\mathbb{R}^{2 n})$ consists of every $f\in C^\infty(\mathbb{R}^{2 n})$ for which there is an $h>0$ such that
     \begin{equation}\label{eq:doubledef}
         \sup \dfrac{\left| x^{\alpha_1} \xi^{\alpha_2} D_x^{\beta_1} D_\xi^{\beta_2} f(x,\xi) \right|}{h^{|\alpha_1 + \beta_2|}(\alpha_1!)^s (\beta_2!)^t} \leq C_{\beta_1,\alpha_2}
     \end{equation}
     for every $\alpha_1,\alpha_2,\beta_1,\beta_2\in \mathbb{N}^n,$ where the supremum is taken over $x,\xi\in\mathbb{R}^n$ and where $C_{\beta_1,\alpha_2}$ is a constant depending only on $\beta_1$ and $\alpha_2$.
    \end{definition}
    In similar ways with $\sigma,\tau>0$, we let $S_{\infty,\tau}^{\sigma,\infty}(\mathbb{R}^{2 n})$ consist of every $f\in C^\infty(\mathbb{R}^{2 n})$ for which there is an $h>0$ such that
     \begin{equation}\label{eq:doubledef2}
         \sup \dfrac{\left| x^{\alpha_1} \xi^{\alpha_2} D_x^{\beta_1} D_\xi^{\beta_2} f(x,\xi) \right|}{h^{|\beta_1+\alpha_2|}(\beta_1!)^{\sigma}(\alpha_2!)^{\tau}} \leq C_{\alpha_1,\beta_2}
     \end{equation}
     for every $\alpha_1,\alpha_2,\beta_1,\beta_2\in \mathbb{N}^n$. Here, the supremum is taken over $x,\xi\in\mathbb{R}^n$ and $C_{\alpha_1,\beta_2}$ is a constant depending only on $\alpha_1$ and $\beta_2$.
    We also consider the duals of these spaces, which we construct as follows.

    Let $||f||_{h,N,M}$ be the supremum in \eqref{eq:doubledef} taken over $x,\xi\in\mathbb{R}^n$, $\alpha_2,\beta_1\in\mathbb{N}^n$, but only $|\alpha_1|\leq N$ and $|\beta_2|\leq M$. With \[ V_{h,N,M}(\mathbb{R}^{2 n}) = \{f\in C^\infty(\mathbb{R}^{2 n}): ||f||_{h,N,M}<\infty \},\] we observe that
    \begin{equation}
        S_{s,\infty}^{\infty,t}(\mathbb{R}^{2 n}) = \indlim_{h>0} \left(\projlim_{N,M\geq 0} V_{h,N,M}(\mathbb{R}^{2 n}) \right).
    \end{equation}
    It is therefore natural to set
    \begin{equation}
        (S_{s,\infty}^{\infty,t})'(\mathbb{R}^{2 n}) = \projlim_{h>0} \left(\indlim_{N,M\geq 0} (V_{h,N,M})'(\mathbb{R}^{2 n}) \right).
    \end{equation}
    We construct the space $(S_{\infty,\tau}^{\sigma,\infty})'(\mathbb{R}^{2 n})$ analogously.
    
    Similar to the characterizations of $S_s$ and $S^\sigma$ via Fourier transform in Theorem \ref{thm:SsFT}, we can characterize the double spaces $S_{s,\infty}^{\infty,t}$ and $S_{\infty,\tau}^{\sigma,\infty}$ as follows.
    
    \begin{prop}\label{prop:doublespace}
    Suppose $s,t,\sigma,\tau>0$ and $f\in C^{\infty}(\mathbb{R}^{2 n})$. 
    \begin{enumerate}[label=(\alph*)]
        \item $f\in S_{s,\infty}^{\infty,t}(\mathbb{R}^{2 n})$ if and only if there is an $r>0$ such that
        \begin{equation}
            |f(x,\xi)|\leq C_N (1+|\xi|^2)^{-N} e^{-r |x|^{1/s}}, \quad |\hat{f}(\eta,y)|\leq (1+|\eta|^2)^{-N} e^{- r|y|^{1/t}}
        \end{equation}
        for every $N\geq 0$.
        \item $f\in S_{\infty,\tau}^{\sigma,\infty}(\mathbb{R}^{2 n})$ if and only if there is an $r>0$ such that
        \begin{equation}
            |f(x,\xi)|\leq C_N (1+|x|^2)^{-N} e^{-r |\xi|^{1/\tau}}, \quad |\hat{f}(\eta,y)|\leq (1+|y|^2)^{-N} e^{- r|\eta|^{1/\sigma}}
        \end{equation}
        for every $N\geq 0$.
    \end{enumerate}
    \end{prop}
    \begin{proof}
    This result follows directly from Theorem \ref{thm:SsFT}.
    \end{proof}
    With this in mind we now consider Toeplitz operators on one-parameter Gelfand-Shilov spaces.
    \begin{definition}\label{def:top}
    Let $s,\sigma>0$, $\phi_1,\phi_2 \in S_s(\mathbb{R}^n)$ and $a\in S_{s,\infty}^{\infty,s}(\mathbb{R}^{2 n})$. 
    The Toeplitz operator $Tp_{\phi_1,\phi_2}(a)$ is given by
     \begin{equation}\label{eq:top}
         (Tp_{\phi_1,\phi_2}(a) f, g)_{L^2(\mathbb{R}^{2 n})} = (a, \overline{V_{\phi_1} f}\cdot V_{\phi_2} g)_{L^2(\mathbb{R}^{2 n})}
     \end{equation}
     for every $f\in S_s(\mathbb{R}^n)$ and $g\in (S_s)'(\mathbb{R}^n)$.
     
     If instead $\phi_1,\phi_2\in S^\sigma(\mathbb{R}^n)$ and $a\in S^{\sigma,\infty}_{\infty,\sigma}(\mathbb{R}^{2 n})$, the Toeplitz operator is given by \eqref{eq:top} for every $f\in S^\sigma(\mathbb{R}^n)$ and $g\in (S^\sigma)'(\mathbb{R}^n)$.
    \end{definition}
    We observe that the Toeplitz operator in \eqref{eq:top} can be expressed as $$ Tp_{\phi_1,\phi_2}(a)f = V_{\phi_2}^*(a \cdot V_{\phi_1} f)$$
    and that this is a continuous operator from $S_s(\mathbb{R}^{n})$ to $S_s(\mathbb{R}^{n})$ when $a\in S_{s,\infty}^{\infty,s}(\mathbb{R}^{2 n})$, and from $S^\sigma(\mathbb{R}^{n})$ to $S^\sigma(\mathbb{R}^{n})$ when $a\in S_{\infty,\sigma}^{\sigma,\infty}(\mathbb{R}^{2 n})$. We now want to show that we can loosen the restriction on $a$ to instead be in the duals of $S_{s,\infty}^{\infty,s}(\mathbb{R}^{2 n})$ and $S_{\infty,\sigma}^{\sigma,\infty}(\mathbb{R}^{2 n})$. To do this, we need the following lemma.
    \begin{lemma}\label{lem:doubleFT}
    Let $\phi_1,\phi_2,f \in S_s(\mathbb{R}^n)$ and $g \in (S_s)'(\mathbb{R}^n)$. Then
    \begin{equation*}
        \mathscr{F} [\overline{V_{\phi_1} f} \cdot V_{\phi_2} g] (\eta,y) = e^{i y \eta} V_{\phi_2} \phi_1 (y,-\eta) \cdot   V_f g(-y,\eta)
    \end{equation*}
    \end{lemma}
    \begin{proof}
    We have
    \begin{align*}
      \mathscr{F} [\overline{V_{\phi_1} f} \cdot V_{\phi_2} g] (\eta,y) &= (2\pi)^{-2 n} \iiiint \overline{f(z)} \phi_1(z-x) g(w) \overline{\phi_2(w-x)} e^{i(z-w-y)\xi - i x\eta} \, d z d w d x d \xi \\ 
    &= (2\pi)^{- n} \iint \overline{f(w+y)} \phi_1(w+y-x) g(w) \overline{\phi_2(w-x)} e^{- i x\eta} \, d w \, d x \\
    &= (2\pi)^{- n} \iint \overline{f(s)} \phi_1(s-x) g(s-y) \overline{\phi_2(s-y-x)} e^{- i x\eta} \, d s \, d x \\
    &= (2\pi)^{- n} \iint \overline{f(s)} \phi_1(t) g(s-y) \overline{\phi_2(t-y)} e^{- i (s-t)\eta} \, d s \, d t \\
    &= (2\pi)^{- n} \iint \overline{f(z+y)} \phi_1(t) g(z) \overline{\phi_2(t-y)} e^{- i (z+y-t)\eta} \, d z \, d t \\
    &= e^{-i y \eta} V_{\phi_2} \phi_1 (y,-\eta) \cdot   V_f g(-y,\eta) \\
    \end{align*}
    where we apply the Fourier inversion theorem in the second step, apply the variable substitution $s = w + y$ in the third step, $t = s - x$ in the fourth step and $z = s - y$ in the fifth step.
    \end{proof}
    With this lemma in mind we move on to the main result of this section.
    \begin{thm}
    Suppose $\phi_1,\phi_2 \in S_s(\mathbb{R}^n)$. Then the following is true.
    \begin{enumerate}[label=(\alph*)]
        \item The definition of $Tp_{\phi_1,\phi_2}(a)$ is uniquely extendable to any $a\in (S_{s,\infty}^{\infty,s})'(\mathbb{R}^{2 n})$ and is then continuous from $S_s(\mathbb{R}^n)$ to $(S_s)'(\mathbb{R}^n)$.
        \item If $a\in (S_{s,\infty}^{\infty,s})'(\mathbb{R}^{2 n})$ then $Tp_{\phi_1,\phi_2}(a)$ is continuous on $S_s(\mathbb{R}^n)$ and uniquely extendable to a continuous operator on $(S_s)'(\mathbb{R}^n)$.
    \end{enumerate}
    \end{thm}
    \begin{proof}
    For (a), it is sufficient to show that $H = \overline{V_{\phi_1} f}\cdot V_{\phi_2} g \in S_{s,\infty}^{\infty,s}(\mathbb{R}^{2 n})$ whenever $f\in S_s(\mathbb{R}^n)$ and $g\in (S_s)'(\mathbb{R}^n)$ and for (b), the same statement is sufficient but with $f\in (S_s)'(\mathbb{R}^n)$ and $g\in S_s(\mathbb{R}^n)$. 
    
    We begin by proving (a). By Theorem \ref{thm:SsSTFT} and Theorem \ref{thm:SsSTFTdual}, for every $N\geq 0$ and every $r>0$, there are $r_0 > 0$, $N_0 \geq 0$ and $C_{N,r} > 0$ such that
    \begin{align*}
        |H(x,\xi)| &= (2\pi)^{-n} | V_{\phi_1} f (x,\xi)|\cdot |V_{\phi_2}g(x,\xi)| \\
        &\leq C_{N,r} (1 + |\xi|^2)^{-(N-N_0)} e^{-(r_0-r) |x|^{1/s}}.
    \end{align*}
    Picking $r<r_0$ and noting that $N_0 \geq 0$ gives the first inequality of Proposition \ref{prop:doublespace}.
    
    By Lemma \ref{lem:doubleFT},
    \begin{align*}
       |\hat{H}(\eta,y)| &= (2\pi)^{-n} | V_{\phi_2} \phi_1 (y,-\eta) \cdot V_f g(-y,\eta)|.
    \end{align*}
    Applying Theorem \ref{thm:SsSTFT} and Theorem \ref{thm:SsSTFTdual} exactly as before, we now obtain the second inequality of Proposition \ref{prop:doublespace}. This completes the proof of (a). To prove (b), simply reverse the roles of $f$ and $g$ and the result follows.
    
    \end{proof}
    We also state the corresponding result for $a\in (S_{\infty,\sigma}^{\sigma,\infty})'(\mathbb{R}^{2 n})$, which follows by similar arguments.
    \begin{thm}
    Suppose $\phi_1,\phi_2 \in S^\sigma(\mathbb{R}^n)$. Then the following is true.
    \begin{enumerate}[label=(\alph*)]
        \item The definition of $Tp_{\phi_1,\phi_2}(a)$ is uniquely extendable to any $a\in (S_{\infty,\sigma}^{\sigma,\infty})'(\mathbb{R}^{2 n})$ and is then continuous from $S^\sigma(\mathbb{R}^n)$ to $(S^\sigma)'(\mathbb{R}^n)$.
        \item If $a\in (S_{\infty,\sigma}^{\sigma,\infty})'(\mathbb{R}^{2 n})$ then $Tp_{\phi_1,\phi_2}(a)$ is continuous on $S^\sigma(\mathbb{R}^n)$ and uniquely extendable to a continuous operator on $(S^\sigma)'(\mathbb{R}^n)$.
    \end{enumerate}
    \end{thm}

\end{document}